\documentclass[11pt,reqno]{amsart}

\usepackage{amsmath}								%math
\usepackage{amssymb}
\usepackage{amsthm}
\usepackage{amscd}
\usepackage{amsfonts}

\usepackage{mathptmx}

\usepackage[colorlinks, linktocpage, citecolor = blue, linkcolor = blue]{hyperref}
\usepackage{color}

\usepackage{tikz}									%tikz
\usetikzlibrary{patterns}
\usetikzlibrary{matrix}
\usetikzlibrary{positioning}
\usepackage{fullpage}
\usepackage{enumerate}

\newtheorem{theorem*}{Theorem}		
\newtheorem{corollary*}{Corollary}		
\newtheorem{theorem}{Theorem}[section]
\newtheorem{lemma}[theorem]{Lemma}
\newtheorem{proposition}[theorem]{Proposition}
\newtheorem{corollary}[theorem]{Corollary} 

\theoremstyle{definition}
\newtheorem{definition}[theorem]{Definition}
\newtheorem{example}[theorem]{Example}

\theoremstyle{remark}
\newtheorem{remark}[theorem]{Remark}
\newtheorem{remarks}[theorem]{Remarks}

\newcommand{\R}{\mathbb{R}}

\newcommand{\Rbar}{\overline{\mathbb{R}}}

\newcommand{\C}{\mathbb{C}}
\newcommand{\CC}{\mathbb{C}}

\newcommand{\A}{\mathbb{A}}
\newcommand{\G}{\mathbb{G}}

\newcommand{\calA}{\mathcal{A}}

\newcommand{\calM}{\mathcal{M}}

\newcommand{\calX}{\mathcal{X}}
\newcommand{\calY}{\mathcal{Y}}
\newcommand{\calZ}{\mathcal{Z}}

\newcommand{\frakS}{\mathfrak{S}}

\newcommand{\bfp}{\mathbf{p}}

\DeclareMathOperator{\Spec}{Spec}
\DeclareMathOperator{\Hom}{Hom}

\DeclareMathOperator{\Isom}{Isom}
\DeclareMathOperator{\trop}{trop}

\DeclareMathOperator{\id}{id}
\DeclareMathOperator{\HOM}{HOM}

\title[Tropicalization is a stack quotient]{Tropicalization is a non-Archimedean analytic stack quotient}
\author{Martin Ulirsch}
\address{Department of Mathematics, Brown University, Providence, RI 02912, USA}
\email{\href{mailto:ulirsch@math.brown.edu}{ulirsch@math.brown.edu}}
\urladdr{\url{http://www.math.brown.edu/~ulirsch/index.html}}
\date{\today}
\thanks{The author's research was supported in part by funds from BSF grant 201025 and NSF grants DMS0901278 and DMS1162367.} 

\begin{document}

\maketitle

\begin{abstract}
For a complex toric variety $X$ the logarithmic absolute value induces a natural retraction of $X$ onto the set of its non-negative points and this retraction can be identified with a quotient of $X(\mathbb{C})$ by its big real torus. We prove an analogous result in the non-Archimedean world: The Kajiwara-Payne tropicalization map is a non-Archimedean analytic stack quotient of $X^{an}$ by its big affinoid torus. Along the way, we provide foundations for a geometric theory of non-Archimedean analytic stacks, particularly focussing on analytic groupoids and their quotients, the process of analytification, and the underlying topological spaces of analytic stacks.
\end{abstract}

\setcounter{tocdepth}{1}
\tableofcontents

%%%%%%%%%%%%%%%%%%%%%%%%%%%%%%%%%%%%%%%%%%%%%%%%%%%%%%

\section{Introduction}

Let $T\simeq \G_m^n$ be a split algebraic torus and denote by $N$ the dual of its character lattice $M$. Suppose that $X=X(\Delta)$ is a $T$-toric variety defined by a rational polyhedral fan $\Delta$ in $N_\R=N\otimes\R$. We refer the reader to \cite{Cox_toric} and \cite{Fulton_toricvarieties} for the standard notation for toric varieties and details of this beautiful theory. 

\subsection{The Archimedean case} Suppose first that $X$ is defined over $\CC$. The logarithmic absolute value on $\CC$ induces a natural continuous map 
\begin{equation*}
X(\CC)\longrightarrow N_\R(\Delta)
\end{equation*}
onto a partial compactification $N_\R(\Delta)$ of $N_\R$, whose fibers are homogenous spaces under the operation of the real torus 
\begin{equation*}
N_{S^1}=N\otimes S^1\subseteq N_{\CC^\ast}=N\otimes\CC^\ast=T(\C)\ .
\end{equation*}

Write $\Rbar=\big(\R\sqcup\{\infty\},+\big)$ with the naturally defined addition. On a $T$-invariant open affine subset $U_\sigma=\Spec \CC[S_\sigma]$ for a cone $\sigma$ in $\Delta$ we have $N_\R(\sigma)=\Hom(S_\sigma,\Rbar)$ with the topology of pointwise convergence and the map is given by
\begin{equation*}\begin{split}
U_\sigma(\CC)=\Hom\big(S_\sigma, (\CC,\cdot)\big)&\longrightarrow N_\R(\sigma)=\Hom(S_\sigma,\Rbar)\\
u&\longrightarrow -\log \vert \cdot\vert\circ u \ .
\end{split}\end{equation*} 
This map admits a continuous section, whose preimage is the locus of non-negative points $X(\CC)_{\geq 0}$ of $X$, and we can therefore reinterpret it as a retraction $X(\CC)\rightarrow X(\CC)_{\geq 0}$.

In fact, one can identify this retraction with the topological quotient map
\begin{equation*}
X(\CC)\longrightarrow X(\CC)/N_{S^1}
\end{equation*}
where the operation of $N_{S^1}$ on $X(\CC)$ is induced by the natural operation of $T$ on $X$ (see \cite[Proposition 12.2.3]{Cox_toric}). Furthermore, if $X$ is projective with a chosen polarization, defined by a lattice polytope $P$ in $M_\R=M\otimes\R$, then there is a natural moment map $X(\CC)\rightarrow P$, whose restriction to $X(\CC)_{\geq 0}$ is a homeomorphism (see \cite[Section 4.2]{Fulton_toricvarieties}). 
 
\subsection{The non-Archimedean case}

Let $k$ be a non-Archimedean field, possibly carrying the trivial norm. Kajiwara \cite{Kajiwara_troptoric} and, independently, Payne \cite{Payne_anallimittrop} have defined a continuous tropicalization map 
\begin{equation*}
\trop_\Delta\mathrel{\mathop:}X^{an}\longrightarrow N_\R(\Delta)
\end{equation*} 
from the non-Archimedean analytic space $X^{an}$ in the sense of Berkovich (see \cite{Berkovich_book} and \cite{Berkovich_etalecoho}) into $N_\R(\Delta)$ that can be identified with a natural deformation retraction onto the \emph{non-Archimedean skeleton} $\frakS(X)$ of $X^{an}$ (see \cite[Section 2]{Thuillier_toroidal} and Lemma \ref{lemma_trop=skeleton} below). In  \cite[Remark 3.3]{Payne_anallimittrop}) it has been suggested that $\trop_\Delta$ is a non-Archimedean version of the moment map, a fact that has been established by Kajiwara in \cite[Theorem 2.2]{Kajiwara_troptoric}, when $X$ is projective.  

Write $T^\circ$ for the \emph{affinoid torus} 
\begin{equation*}
T^\circ=\big\{x\in T^{an}\big\vert \vert\chi^m\vert_x=1 \textrm{ for all $m\in M$ }\big\} \ , 
\end{equation*} 
an analytic subgroup of the analytic group $T^{an}$ that forms the natural analogue of $N_{S^1}$ in the non-Archimedean world. The torus operation $T\times X\rightarrow X$ induces an operation of $T^{an}$, and therefore of $T^\circ$, on $X^{an}$. Unfortunately we cannot take the quotient of $X^{an}$ by $T^\circ$ in the category of topological spaces, since the underlying set of $T^\circ$ does not admit a group structure. 

In Section \ref{section_stacks} we work out foundations for a geometric theory of \emph{non-Archimedean analytic stacks}, geometric stacks over the category of non-Archimedean analytic spaces in the sense of Berkovich, which allows us to take such quotients. Based on this framework we develop in Section \ref{section_topology} the notion of an underlying topological space $\vert\calX\vert$ of a non-Archimedean analytic stack $\calX$. Using this language we prove the following Theorem \ref{thm_trop=quot} identifying the tropicalization map $\trop_\Delta$ with the non-Archimedean analytic stack quotient $X^{an}$ by $T^\circ$, in complete analogy with the corresponding result in the Archimedean case. 

\begin{theorem}\label{thm_trop=quot} There is a natural homeomorphism $\mu_\Delta\mathrel{\mathop:}\big\vert[X^{an}/T^\circ]\big\vert\xrightarrow{\sim}N_\R(\Delta)$ that makes the diagram
\begin{center}\begin{tikzpicture}
  \matrix (m) [matrix of math nodes,row sep=2em,column sep=3em,minimum width=2em]
  {  
  &X^{an}&  \\ 
  \big\vert[X^{an}/T^\circ]\big\vert&  &N_\R(\Delta)  \\ 
  };
  \path[-stealth]
    (m-1-2) edge node [above right] {$\trop_\Delta$} (m-2-3)
    		edge (m-2-1)
    (m-2-1) edge node [below] {$\mu_{\Delta}$} node [above] {$\sim$} (m-2-3);		
\end{tikzpicture}\end{center}
commute. 
\end{theorem}

In other words, on the level of underlying topological spaces, the Kajiwara-Payne tropicalization map $\trop_\Delta$ and the analytic stack quotient map $X\rightarrow[X^{an}/T^{\circ}]$ are equal. Note that by Proposition \ref{prop_topanalstack} below Theorem \ref{thm_trop=quot} implies the well-known fact that the tropicalization map $\trop_\Delta\mathrel{\mathop:}X^{an}\rightarrow N_\R(\Delta)$ is a topological quotient map, which also follows from the properness of $\trop_\Delta$ (see \cite[Proposition 2.1 and Section 3]{Payne_anallimittrop}). 

In particular, Theorem \ref{thm_trop=quot} says that $N_\R(\Delta)$, a purely combinatorial object that serves as a tropical analogue of a toric variety, canonically admits the structure of a non-Archimedean analytic stack. Based on this observation, one is led to speculate that tropical geometry can be axiomatized as the geometry of "affinoid substacks" of $\big[X^{an}/T^\circ\big]$. The authors hopes to return to this speculation at some later point, once the necessary theory of non-Archimedean analytic stacks has been developed, and to make this statement precise.  

It is worth noting that the operation of $T^\circ$ on $X^{an}$ already lies at the very heart of the construction of the non-Archimedean skeleton $\frakS(X)$ of $X^{an}$, as explained in \cite[Theorem 6.1.5]{Berkovich_book} for $\mathbb{P}^n$ and \cite[Section 2]{Thuillier_toroidal} in the case of $k$ carrying the trivial absolute value. In fact, our proof of Theorem \ref{thm_trop=quot} (see Section \ref{section_trop=quot}) essentially goes by showing that the skeleton $\frakS(X)$ of $X^{an}$ is the set of $T^\circ$-invariant points of $X^{an}$.

\begin{example}\label{example_A1/Gm}
Consider the affine line $\A^1$ over a trivially valued field $k$. The non-Archimedean unit circle $\G_m^\circ$ is given as the subset of elements in $x\in(\A^1)^{an}$ with $\vert t\vert_x=1$, where $t$ denotes a coordinate on $\A^1$. The skeleton $\frakS(\A^1)$ of $(\A^1)^{an}$ is the line connecting $0$ to $\infty$. It is precisely the set of "$\G_m^\circ$-invariant" points in $(\A^1)^{an}$ and therefore naturally homeomorphic to the topological space underlying $\big[(\A^1)^{an}\big/\G_m^\circ\big]$.

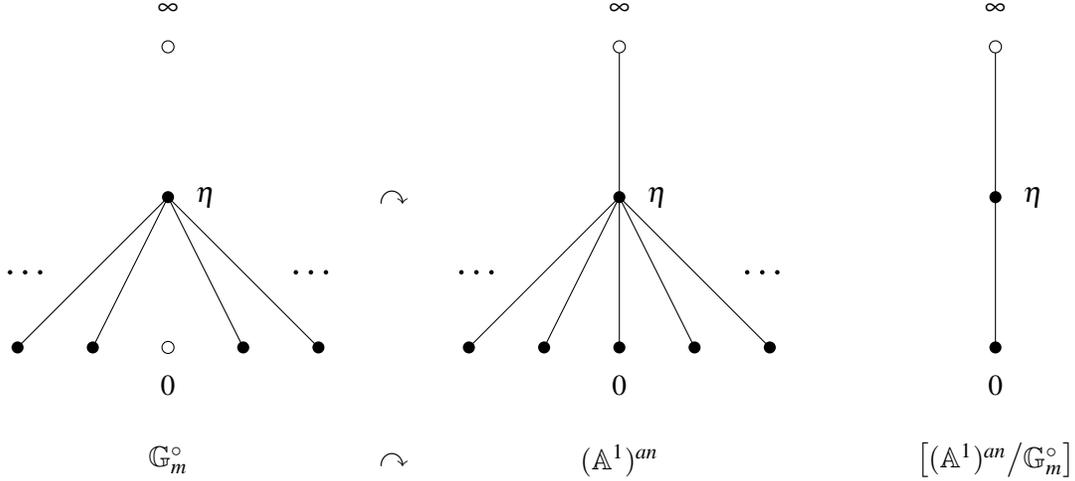
\begin{center}
\begin{tikzpicture}
\draw (6,2) circle (0.08 cm);
\fill (6,0) circle (0.08 cm);
\fill (4,-2) circle (0.08 cm);
\fill (5,-2) circle (0.08 cm);
\fill (6,-2) circle (0.08 cm);
\fill (7,-2) circle (0.08 cm);
\fill (8,-2) circle (0.08 cm);

\draw (6,0) -- (6,1.92);
\draw (6,0) -- (4,-2);
\draw (6,0) -- (5,-2);
\draw (6,0) -- (6,-2);
\draw (6,0) -- (7,-2);
\draw (6,0) -- (8,-2);

\node at (6,-2.5) {$0$};
\node at (6,2.5) {$\infty$};
%\node at (8.5,-2) {$\A^1$};
\node at (6.5,0) {$\eta$};
\node at (6,-3.5) {$(\A^1)^{an}$};

\fill (4.3,-1) circle (0.03cm);
\fill (4.1,-1) circle (0.03cm);
\fill (3.9,-1) circle (0.03cm);
\fill (7.7,-1) circle (0.03cm);
\fill (7.9,-1) circle (0.03cm);
\fill (8.1,-1) circle (0.03cm);

\draw (0,2) circle (0.08 cm);
\fill (0,0) circle (0.08 cm);
\fill (-2,-2) circle (0.08 cm);
\fill (-1,-2) circle (0.08 cm);
\draw (0,-2) circle (0.08 cm);
\fill (1,-2) circle (0.08 cm);
\fill (2,-2) circle (0.08 cm);

%\draw (0,0) -- (0,1.92);
\draw (0,0) -- (-2,-2);
\draw (0,0) -- (-1,-2);
%\draw (0,0) -- (0,-2);
\draw (0,0) -- (1,-2);
\draw (0,0) -- (2,-2);

\node at (0,-2.5) {$0$};
\node at (0,2.5) {$\infty$};
%\node at (2.5,-2) {$\G_m$};
\node at (0.5,0) {$\eta$};
\node at (0,-3.5) {$\G_m^{\circ}$};

\fill (-1.7,-1) circle (0.03cm);
\fill (-1.9,-1) circle (0.03cm);
\fill (-2.1,-1) circle (0.03cm);
\fill (1.7,-1) circle (0.03cm);
\fill (1.9,-1) circle (0.03cm);
\fill (2.1,-1) circle (0.03cm);

\fill (11,0) circle (0.08cm);
\fill (11,-2) circle (0.08cm);
\draw (11,2) circle (0.08cm);

\draw (11,0) -- (11,-2);
\draw (11,0) -- (11,1.92);

\node at (11.5,0) {$\eta$};
\node at (11,-3.5) {$\big[(\A^1)^{an}\big/\G_m^\circ\big]$};
\node at (11,-2.5) {$0$};
\node at (11,2.5) {$\infty$};

\node at (3,0) {$\curvearrowright$};
\node at (3,-3.5) {$\curvearrowright$};

\end{tikzpicture}
\end{center}
\end{example}

Example \ref{example_A1/Gm} also illustrates why it is important to take quotients with respect to the operation of the affinoid torus $T^\circ$ instead of $T^{an}$. The topological space underlying $[X^{an}/T^{an}]$ is homeomorphic to $\big\vert[X/T]\big\vert$, whose points correspond to the $T$-orbits in $X$. The topology on $[X/T]$ is determined by the poset structure on the set of $T$-orbits that is given by containment of orbit closures in $X$, and in particular not Hausdorff. 

In order to get the gist of the argument in the proof of Theorem \ref{thm_trop=quot} the reader may want to mostly skip the very technical Section \ref{section_stacks}, only referring to the necessary definitions when needed and taking some results, such as Proposition \ref{prop_groupoidquot=analstack}, as a black box. 

\subsection{Applications}

\subsubsection{Non-Archimedean geometry of Artin fans}
In \cite{Ulirsch_nonArchArtin} (also see \cite[Section V]{Ulirsch_thesis}) we study the non-Archimedean analytic geometry of \emph{Artin fans}, certain locally toric Artin stacks that have been introduced in \cite{AbramovichChenMarcusWise_boundedness} and \cite{AbramovichWise_invlogGromovWitten} (also see \cite{AbramovichChenMarcusUlirschWise_logsurvey}). The goal is to relate this theory to the tropical geometry of logarithmic schemes, as introduced in \cite{Ulirsch_functroplogsch}. 

Let $k$ be an algebraically closed field that is endowed with the trivial absolute value. By \cite[Proposition 3.1.1]{AbramovichChenMarcusWise_boundedness} every fine and saturated logarithmic scheme $X$, locally of finite type over $k$, admits a canonical strict morphism $X\rightarrow\calA_X$ into an Artin fan $\calA_X$. We show that on the level of underlying topological spaces the analytic morphism
\begin{equation}\label{eq_Artinfanbeth}
X^\beth\longrightarrow \calA_X^\beth
\end{equation}
is nothing but the tropicalization map of the logarithmic scheme $X$ constructed in \cite{Ulirsch_functroplogsch}, where $(.)^\beth$ is Thuillier's analytic generic fiber functor over trivially valued fields (see \cite[Proposition et D\'efinition 1.3]{Thuillier_toroidal}, \cite[Section 5]{Ulirsch_nonArchArtin}, and \cite[Section V.3]{Ulirsch_thesis}). 

Theorem \ref{thm_trop=quot} is a first instance of this connection that is of independent interest. The Artin fan of a $T$-toric variety $X$ is the toric quotient stack $\calA_X=[X/T]$. If $X$ is complete, then Theorem \ref{thm_trop=quot} says that on the level of underlying topological spaces the analytic stack quotient map
\begin{equation*}
X^{an}=X^\beth\longrightarrow \calA_X^\beth=[X^\beth/T^\beth]
\end{equation*}
is nothing but the tropicalization map $\trop_\Delta$ of $X$.

The identification of \eqref{eq_Artinfanbeth} with the tropicalization map, or, in the logarithmically smooth case, with the natural deformation retraction of $X^\beth$ onto its toroidal skeleton (see \cite{Thuillier_toroidal} and \cite[Theorem 1.2]{Ulirsch_functroplogsch}) lies at the very heart of recent results of Ranganathan \cite{Ranganathan_ratcurvesontorvars}, making explicit the relationship between the tropical and non-Archimedean geometry of moduli spaces (see \cite{AbramovichCaporasoPayne_tropicalmoduli}), tropical enumerative geometry (see \cite{Mikhalkin_enumerativetrop} and \cite{NishinouSiebert_toricdeg&tropcurves}), and logarithmic Gromov-Witten Theory (see \cite{Chen_stablelogmaps} and \cite{GrossSiebert_logGromovWitten}). 

\subsubsection{Realizability of tropical curves over Artin fans}
In \cite{Ranganathan_Artinfanrealizability} the author describes another application of Theorem \ref{thm_trop=quot} to the \emph{realizability problem} for tropical curves by algebraic curves. In general not every tropical curve $\Gamma$ in $N_\R(\Delta)$ arises as the tropicalization of an algebraic curve in $X$. The main reason for this behavior is the phenomenon of \emph{superabundance}, a cohomological obstruction to the existence of deformations of maps from a logarithmically smooth curve into a toric variety that has an interpretation purely in terms of the combinatorial geometry of $\Gamma$ (see \cite[Section 2.6]{Mikhalkin_enumerativetrop}, \cite{Speyer_thesis}, and \cite[Section 1]{Katz_lifting}, as well as \cite[Section 4]{CheungFantiniParkUlirsch_faithfulrealizability}). However, this cohomological obstruction vanishes for maps from a logarithmically smooth curve into the Artin fan $\calA_X=[X/T]$ of $X$ leading to the main result of \cite{Ranganathan_Artinfanrealizability} that every tropical curve can be realized as a curve mapping to the Artin fan $\calA_X$. 

%\subsubsection{"Affinoid substacks" and the axiomatization of tropical geometry}

\subsection{An alternative approach to analytic stacks}

A theory of non-Archimedean analytic stacks similar to ours has already been outlined in \cite[Section 6.1]{Yu_Gromovcompactness}, in the context of non-Archimedean analytic Gromov-Witten Theory, and further developed in \cite{PortaYu_analyticnstacks}. The main difference from our approach is that the above authors work over the category $\mathbf{Rig}$ of quasi-separated rigid analytic spaces with locally finite admissible affinoid coverings, endowed with the Tate-\'etale topology, while, in this article, we work with the category of non-Archimedean analytic spaces  in the sense of Berkovich \cite{Berkovich_etalecoho}, endowed with the \'etale topology constructed in \cite[Section 4]{Berkovich_etalecoho}. 

The category $\mathbf{Rig}$ is equivalent to the category of paracompact strictly $k$-analytic spaces. Since \'etale morphisms are also Tate-\'etale, an analytic stack in our setting is automatically an analytic stack in the setting of \cite{Yu_Gromovcompactness} and \cite{PortaYu_analyticnstacks}. Moreover, since the Tate-\'etale topology is finer than the \'etale topology  in the sense of \cite{Berkovich_etalecoho} (as it includes e.g. closed affinoid domains), Porta and Yue Yu's class of analytic stacks is strictly bigger than ours. This distinction is particularly relevant when studying generic fiber functors, such as Thuillier's $(.)^\beth$-functor (see \cite[Section 5]{Ulirsch_nonArchArtin}). 

In Porta and Yue Yu's setup one can associate to a non-Archimedean analytic stack an underlying topological space in complete analogy with Section \ref{section_topology}. For many important examples, such as the toric quotients in Theorem \ref{thm_trop=quot}, or analytifications of algebraic stacks, both definitions lead to the same underlying topological spaces. 

%%%%%%%%%%%%%%%%%%%%%%%%%%%%%%%%%%%%%%%%%%%%%%%%%%%%%%

\subsection{Conventions and prerequisites}

We denote the category of non-Archimedean analytic spaces in the sense of \cite{Berkovich_etalecoho} by $\mathbf{An}_k$. Given an analytic space $S$, we denote by $(\mathbf{An}_k/S)$ the category of analytic spaces over $S$. A surjective morphism $f\mathrel{\mathop:}X\rightarrow Y$ of analytic spaces is said to be \emph{universally submersive}, if every base change of $f$ is \emph{submersive}, i.e. a topological quotient map.

Following Ducros \cite[Section 3.1]{Ducros_familiesBerkovich} a morphism $f\mathrel{\mathop:}X\rightarrow Y$ between good analytic spaces is said to be \emph{naively flat}, if for all $x\in X$ and $y=\phi(x)$ the $\mathcal{O}_{Y,y}$-algebra $\mathcal{O}_{X,x}$ is flat. As seen in \cite[Section 3.4]{Ducros_familiesBerkovich} this notion is, in general, not preserved under base change. As a solution to this issue, Ducros \cite[Section 3.1.4.2]{Ducros_familiesBerkovich} defines a morphism $f\mathrel{\mathop:}X\rightarrow Y$ to be \emph{universally flat} (or short: \emph{flat}), if all of its good base changes are naively flat. It is an immediate consequence of this definition that being flat is stable under base change. So, in particular, all analytic domains in $X$ are flat over $X$. 

This shows that, in contrast to the category of schemes, not all flat morphisms are open maps. For quasi-finite morphisms, however, this notion of flatness agrees with the one introduced in \cite[Section 3.2]{Berkovich_etalecoho}. Therefore by \cite[Proposition 3.27]{Berkovich_etalecoho} a quasi-finite flat morphism $f\mathrel{\mathop:}X\rightarrow Y$ of analytic spaces is open. 

A morphism $f\mathrel{\mathop:}X\rightarrow Y$ is said to be \emph{$G$-smooth} of relative dimension $n$, if it is flat and the sheaf of relative differentials $\Omega_{X_G/Y_G}$ in the sense of \cite[Section 3.3]{Berkovich_etalecoho} is locally free of dimension $n$. Accordingly, a $G$-smooth morphism of relative dimension $0$ is called \emph{$G$-\'etale}. A morphism $f\mathrel{\mathop:}X\rightarrow Y$ is said to be \emph{\'etale}, if it is quasi-finite and $G$-\'etale. This definition is equivalent to \cite[Definition 3.3.4]{Berkovich_etalecoho}. 

Let $X$ be an analytic space. As defined in \cite[Section 4.1]{Berkovich_etalecoho}, an \emph{\'etale covering} of an analytic space $U$ over $X$ is given by a family of \'etale morphisms $(f\mathrel{\mathop:}U_i\rightarrow U)$ such that $\bigcup_i f(U_i)=U$. The class of \'etale coverings defines a Grothendieck topology on $\mathbf{An}_k$, called the \emph{\'etale topology}. We denote the resulting \emph{\'etale sites} by $(\mathbf{An}_k)_{et}$ and, more generally, for an analytic space $S$ by $(\mathbf{An}_k/S)_{et}$.

We refer the reader to \cite[Section 2]{ConradTemkin_descent} for a long list of properties of morphisms in $\mathbf{An}_k$ that can be checked on an \'etale covering of the target. For our purposes it is enough to keep in mind that this list includes flat, $G$-smooth, $G$-\'etale, \'etale, and surjective morphisms. With the same methods as the ones employed in the proofs of \cite[Theorem 2.4 and 2.5]{ConradTemkin_descent} one can show that whether a morphism is flat, $G$-smooth, $G$-\'etale, \'etale, or surjective can also be checked on an \'etale covering of the domain. An easy argument shows that for a surjective morphism to be universally submersive can be verified on an \'etale covering of the target as well as on an \'etale covering of the domain.

%%%%%%%%%%%%%%%%%%%%%%%%%%%%%%%%%%%%%%%%%%%%%%%%%%%%%%

\subsection{Acknowledgements}
The author would like to express his gratitude to Dan Abra\-movich for his constant support and encouragement. Thanks are also due to Johan de Jong and his collaborators for creating the Stacks Project \cite{stacks-project}, to Sam Payne for many comments an earlier version of this article, to Martin Olsson for sharing a draft of his upcoming book \cite{Olsson_algspaces&stacks}, to Michael Temkin for his advice concerning the descent theory of non-Archimedean analytic spaces, and to the anonymous referee for many insightful comments. Particular thanks are due to Brian Conrad for finding an inaccuracy in an earlier version of Proposition \ref{prop_etaleequivquot=etaleanalspace} and for subsequently suggesting the statement of Lemma \ref{lemma_quotientsofetalecovers} to the author. During the work on this article the author also profited from discussions with Matt Baker, Dori Bejleri, Joseph Rabinoff, Dhruv Ranganathan, and Tony Yue Yu, all of whom he would like to thank heartfully. 

%%%%%%%%%%%%%%%%%%%%%%%%%%%%%%%%%%%%%%%%%%%%%%%%%%%%%%

\section{A geometric theory of non-Archimedean analytic stacks}\label{section_stacks}

The purpose of this section is to lay the foundations for a theory of \emph{geometric stacks}, in the sense of \cite{Simpson_nstacks}, over the category of non-Archimedean analytic spaces in the sense of Berkovich (see \cite{Berkovich_book} and \cite{Berkovich_etalecoho}). A central role in this theory is played by the notion of \emph{analytic groupoids}, groupoid objects in $\mathbf{An}_k$, and their quotient stacks. Using these techniques we construct an analytification pseudo-functor that associates to an algebraic stack $\calX$, locally of finite type over $k$, an analytic stack $\calX^{an}$.   

We freely use the language of categories fibered in groupoids and stacks over arbitrary sites, as developed in  \cite{Giraud_cohomologienonabelienne} and \cite{Vistoli_stacks}, and follow the notations and conventions of the Stacks Project \cite{stacks-project}. The only major difference from this edifice is that for algebraic stacks we are using the big \'etale site over the category of schemes that are locally of finite type over $k$ as an underlying site and not the fppf-site as in \cite[Tag 026O]{stacks-project}. Both approaches are equivalent by \cite[Tag 04X1]{stacks-project}.

%%%%%%%%%%%%%%%%%%%%%%%%%%%%%%%%%%%%%%%%%%%%%%%%%%%%%%

\subsection{\'Etale analytic spaces and analytic stacks}\label{section_etanspaces&analstacks}

Given an analytic space $X$ its associated functor of points is given by
\begin{equation*}\begin{split}
h_X\mathrel{\mathop:}(\mathbf{An}_k)^{op}&\longrightarrow \mathbf{Sets}\\
T&\longmapsto X(T)=\Hom(T,X) \ .
\end{split}\end{equation*}
By Yoneda's Lemma the association $X\mapsto h_X$ faithfully embeds $\mathbf{An}_k$ into the category of pre-sheaves on $\mathbf{An}_k$ as full subcategory. In the following we may therefore safely identify $X$ with $h_X$.
A pre-sheaf on $\mathbf{An}_k$ is said to be \emph{representable} by an analytic space $X$, if it is isomorphic to $h_X$. A morphism $X\rightarrow Y$ of pre-sheaves  on $\mathbf{An}_k$ is said to be \emph{representable}, if for every morphism $T\rightarrow X$ from an analytic space $T$ the base change $X\times_YT$ is representable by an analytic space $S$.

\begin{definition}\label{def_etaleanalyticspace}
An \emph{\'etale analytic space} is a sheaf
\begin{equation*}
X\mathrel{\mathop:}(\mathbf{An}_k)_{et}^{op}\longrightarrow \mathbf{Sets}
\end{equation*}
such that there is an analytic space $U$ together with a representable morphism $U\rightarrow X$ that is surjective and \'etale. 
\end{definition}

We are going to refer to the representable surjective \'etale morphism $U\rightarrow X$ (and in a slight abuse of notation also to $U$ itself) as an \emph{atlas} of the \'etale analytic space $X$. The category of \'etale analytic spaces is the full subcategory of the category of pre-sheaves $(\mathbf{An}_k)^{op}\rightarrow \mathbf{Sets}$ whose objects are \'etale analytic spaces. 

\begin{example}
Let $X$ be an analytic space. In order to show that $X$ is an \'etale analytic space the only non-trivial fact is that  $h_X$ is a sheaf in the \'etale topology. By \cite[Proposition 4.1.3]{Berkovich_etalecoho} this is true when $X$ is a good analytic space and the general case follows from \cite[Theorem 4.1.2]{ConradTemkin_algspaces}, as explained in \cite[Remark 4.1.5]{Berkovich_etalecoho}.
\end{example}

A category fibered in groupoids over $\mathbf{An}_k$ is said to be \emph{representable by an \'etale analytic space $X$}, if it is equivalent to $(\mathbf{An}_k/X)$, the category of analytic spaces over $X$. In this case we again  identify $X$ with $(\mathbf{An}_k/X)$, which is justified by the $2$-Yoneda Lemma (see \cite[Tag 04SS]{stacks-project}). A morphism $\calX\rightarrow\calY$ of categories fibered in groupoids is \emph{representable by \'etale analytic spaces}, if for every morphism $T\rightarrow\calY$ from an analytic space $T$ the $2$-fiber product $\calX\times_\calY T$ is representable by an \'etale analytic space. 

\begin{lemma}\label{lemma_diag=rep}
Let $\mathcal{X}$ be a category fibered in groupoids over $\mathbf{An}_k$. The following properties are equivalent:
\begin{enumerate}[(i)]
\item The diagonal morphism $\Delta_\mathcal{X}\mathrel{\mathop:}\mathcal{X}\rightarrow \mathcal{X}\times \mathcal{X}$ is representable by \'etale analytic spaces.
\item For every analytic space $T$ and any two objects $x,y\in\mathcal{X}(T)$ the presheaf $\Isom_\mathcal{X}(x,y)$ is representable by an \'etale analytic space.
\item Every morphism $U\rightarrow \mathcal{X}$ from an analytic space $U$ to $\mathcal{X}$ is representable by \'etale analytic spaces.
\end{enumerate}
\end{lemma}

The proof of Lemma \ref{lemma_diag=rep} is a simple adaption of \cite[Tag 045G]{stacks-project} to the analytic situation (also see \cite[Proposition 7.13]{Vistoli_intersectiontheorystacks}) and is left to the reader.

\begin{definition}\label{def_analstack}
A stack $\calX$ over $(\mathbf{An}_k)_{et}$ is said to be \emph{analytic}, if the following two axioms hold: 
\begin{enumerate}[(i)]
\item The diagonal morphism $\Delta\mathrel{\mathop:}\mathcal{X}\rightarrow \mathcal{X}\times \mathcal{X}$ is representable by \'etale analytic spaces.
\item There is an analytic space $U$ and a morphism $U\rightarrow \mathcal{X}$ that is $G$-smooth, surjective, and universally submersive. 
\end{enumerate}
\end{definition}

We are going to refer to the morphism $U\rightarrow\mathcal{X}$ (and in a slight abuse of notation to $U$ itself) as an \emph{atlas} of $\mathcal{X}$. Note that by Lemma \ref{lemma_diag=rep} the diagonal morphism $\Delta\mathrel{\mathop:}\mathcal{X}\rightarrow \mathcal{X}\times \mathcal{X}$ being representable by \'etale analytic spaces implies that the atlas $U\rightarrow\mathcal{X}$ is representable by \'etale analytic spaces. If the atlas $U\rightarrow \calX$ can be chosen to be \'etale, we are going to refer to $\calX$ as an \emph{analytic Deligne-Mumford stack}. The \emph{$2$-category of analytic stacks} is defined to be the full subcategory of the $2$-category of categories fibered in groupoids over $(\mathbf{An}_k)_{et}$ whose objects are analytic stacks. 

\begin{example}
Let $X$ be an \'etale analytic space. Since $X$ is a sheaf in the \'etale topology, the category fibered in groupoids $(\mathbf{An}_k/X)$ is a stack over the \'etale site $(\mathbf{An}_k)_{et}$. The stack $X$ is an analytic Deligne-Mumford stack, because the surjective \'etale morphism $U\rightarrow X$ from an analytic space $U$ as in  Definition \ref{def_etaleanalyticspace} forms an atlas of $X$. 
\end{example}

\begin{remarks}\label{remark_etanspacerepdiag}\begin{enumerate}[(i)]
\item In Definition \ref{def_etaleanalyticspace} we only require the atlas $U\rightarrow X$ to be representable. From a theoretical point of view, it would be more pleasing to require every morphism $S\rightarrow X$ from an analytic space $S$ to be representable, or equivalently that the diagonal morphism $\Delta_X\mathrel{\mathop:}X\rightarrow X\times X$ is representable (see \cite[Tag 0024]{stacks-project}). Nevertheless, in this case, the proof of Proposition \ref{prop_etaleequivquot=etaleanalspace} below would require a bootstrap argument similar to \cite[Proposition A.1.1]{ConradLieblichOlsson_Nagataalgspaces} and \cite[Tag 0264]{stacks-project}, which uses more sophisticated techniques from descent theory. Unfortunately analytic analogues of these result do not seem to have appeared in the literature so far. 
\item In both Definition \ref{def_etaleanalyticspace} and Definition \ref{def_analstack} it would be theoretically more appealing to use an analogue of the fppf-topology on the category $\mathbf{An}_k$. Since, to the best of the author's knowledge, there is no consensus in the literature on the nature of this analogue, let alone a satisfying collection of descent theoretic results, we refrain from carrying out this approach. 
\end{enumerate}
\end{remarks}

%%%%%%%%%%%%%%%%%%%%%%%%%%%%%%%%%%%%%%%%%%%%%%%%%%%%%%

% Moreover, in Section \ref{section_Morita} we characterize those analytic groupoids that give rise to equivalent analytic stacks in terms of Morita equivalence. 

%%%%%%%%%%%%%%%%%%%%%%%%%%%%%%%%%%%%%%%%%%%%%%%%%%%%%%

\subsection{Groupoid presentations}\label{section_groupoids}

The goal of this section is to study presentations of \'etale analytic spaces by \'etale equivalence relations and  presentations of analytic stacks by analytic groupoids.

\subsubsection{\'Etale equivalence relations}

\begin{definition} Let $U$ be an analytic space. An \emph{\'etale equivalence relation} on $U$ consists of a monomorphism $R\hookrightarrow U\times U$ such that
\begin{enumerate}[(i)]
\item for all analytic spaces $T$ the subset $R(T)\subseteq U(T)\times U(T)$ defines an equivalence relation, and
\item the compositions $R\hookrightarrow U\times U\rightrightarrows U$ are \'etale.
\end{enumerate}
\end{definition}

Given an \'etale equivalence relation $R$ on an analytic space $U$, the association
\begin{equation*}
T\longmapsto U(T)/R(T) \ .
\end{equation*}
defines a pre-sheaf $U/_{pre}R$ on $\mathbf{An}_k$. We refer to the sheafification $U/R$ of $U/_{pre}R$ on $(\mathbf{An}_k)_{et}$ as the \emph{quotient} of $U$ by $R$. 

\begin{proposition}\label{prop_etaleequivquot=etaleanalspace}
Let $R$ be an \'etale equivalence relation on an analytic space $U$. Then the quotient $U/R$ is an \'etale analytic space. 
\end{proposition}

Our approach to the proof of Proposition \ref{prop_etaleequivquot=etaleanalspace} is inspired by \cite[Tag 0264]{stacks-project}, but only uses the descent-theoretic results contained in \cite{ConradTemkin_algspaces} (see Remark \ref{remark_etanspacerepdiag} above). The statement of the following Lemma \ref{lemma_quotientsofetalecovers} has been communicated to the author by Brian Conrad. We would like to thank him for generously allowing us to include it in this article.%\footnote{Possible mistakes are of course still fully within the author's responsibility.}

Let $S$ be a fixed analytic space. One can easily generalize the above notion to define an \'etale equivalence relations $R\hookrightarrow U\times_S U$ in the category of analytic spaces over $S$. In a slight abuse of notation its quotient sheaf over $\big(\mathbf{An}_k/S\big)_{et}$ will also be denoted by $U/R$. 

\begin{lemma}\label{lemma_quotientsofetalecovers}
Let $R\hookrightarrow U\times_S U$ an \'etale equivalence relation in the category of analytic spaces over $S$. If $U\rightarrow S$ is \'etale, then the quotient sheaf $U/R$ over $\big(\mathbf{An}_k/S\big)_{et}$ is representable by an analytic space $X$, which is \'etale over $S$. 
\end{lemma}

\begin{proof}
Let $U_i\subseteq U$ be a cover of $U$ by open subsets. Set $R_i=R\times_{U\times U}(U_i\times U_i)=R\cap (U_i\times U_i)$. Then the quotient $U/R$ is representable by an analytic space $X$ \'etale over $S$ if and only if all quotients $U_i/R_i$ are representable by analytic spaces $X_i$ \'etale over $S$. This statement is an immediate generalization of \cite[Lemma 4.2.3]{ConradTemkin_algspaces} and its proof is the same in our situation. Since every point in $U$ has an open neighborhood that is a finite \'etale cover of an open subset of $S$, we can therefore assume that in our claim the \'etale morphism $U\rightarrow S$ is finite \'etale. 

Now consider the base change $U'=U\times_SU$ of $U$ along $U\rightarrow S$. Setting $R'=R\times_U U'$ we obtain an \'etale equivalence relation $R'\hookrightarrow U'\times_S U'$ whose quotient $U'/R'$ is representable by an analytic space $X'$, since the morphisms $R'\rightrightarrows U'$ admit sections. Similarly, we can consider the base change $U''=U'\times_U U'$ as well as the induced \'etale equivalence relation $R''\hookrightarrow U''\times_S U''$ with $R''=R'\times_{U'}U''$. The morphisms $R''\rightrightarrows U''$ again admit sections and therefore the quotient $U''/R''$ is representable by an analytic space $X''$. We have an induced \'etale equivalence relation $X''\hookrightarrow X'\times_SX'$ whose diagonal is a finite monomorphism, as a base change of $U\rightarrow S$, and thus a closed immersion. By a relative version of \cite[Theorem 1.2.2]{ConradTemkin_algspaces} the quotient $X''/X'$ is representable by an analytic space $X$ over $S$. Finally, the morphism $X\rightarrow S$ is \'etale, since $X'\rightarrow X$ is \'etale and surjective and the composition $X'\rightarrow X\rightarrow S$ is \'etale. 
\end{proof}

\begin{proof}[Proof of Proposition \ref{prop_etaleequivquot=etaleanalspace}]
Write $X=U/R$. Let $T\rightarrow X$ be a morphism from an analytic space $T$ to the quotient sheaf $X=U/R$. We have to show that $Z=T\times_X U$ is representable by an analytic space. 

There is an \'etale covering $(T_i\rightarrow T)$ of $T$ such that $X\vert_{T_i}=(U/_{pre}R)\vert_{T_i}$. In this case the morphisms $T_i\rightarrow X$ factor through morphisms $T_i\rightarrow U$ and the morphisms $T_{ij}=T_i\times_T T_i\rightarrow U\times U$ factor through morphisms $T_{ij}\rightarrow R$. In this case we have natural isomorphisms
\begin{equation*}\begin{split}
T_i\times_{T}Z&\simeq T_i\times_T T\times_X U\\
&\simeq T_i\times_{X}U\\
&\simeq T_i\times_U U\times_XU\\
&\simeq T_i\times_U R \ .
\end{split}\end{equation*}
Therefore $T_i\times_TZ$ is representable by an analytic space $S_i$ and the morphisms $S_i\rightarrow T_i$ are \'etale and surjective, since $U\rightarrow X$ is \'etale and surjective. The pullback of the effective descent datum $\big(T_i, \phi_{ij}\mathrel{\mathop:}T_{ij}\xrightarrow{\sim} T_{ji}\big)$ via $Z\rightarrow T$ induces a descent datum over $Z$. This descent datum is effective by Lemma \ref{lemma_quotientsofetalecovers}, since the morphisms $S_i\rightarrow T_i$ are \'etale and surjective. Therefore $Z$ is representable by an analytic space $S$. 

Finally, the morphism $S\rightarrow T$ is \'etale and surjective, since the $S_i\rightarrow T_i$ are \'etale and surjective. Thus $U\rightarrow X$ is \'etale and surjective as well.
\end{proof}

A \emph{presentation} of an \'etale analytic space $X$ is given by an \'etale equivalence relation $R$ on an analytic space $U$ together with an isomorphism $U/R\simeq X$. The following Proposition \ref{prop_atlas=etaleequivalence} shows that every \'etale analytic space has a presentation.

\begin{proposition}\label{prop_atlas=etaleequivalence}
Let $X$ be an \'etale analytic space and $f\mathrel{\mathop:}U\rightarrow X$ be a representable surjective \'etale morphism from an analytic space $U$ onto $X$. Set $R=U\times_X U$. Then the monomorphism $R\hookrightarrow U\times U$ defines an \'etale equivalence relation and the morphism $U\rightarrow X$ induces an isomorphism $U/R\simeq X$. 
\end{proposition}

Our proof of Proposition \ref{prop_atlas=etaleequivalence} is simple adaption of \cite[Tag 0262]{stacks-project}.

\begin{proof}[Proof of Proposition \ref{prop_atlas=etaleequivalence}]
For an analytic space $T$ we have 
\begin{equation*}
R(T)=\big\{(a,b)\in U(T)\times U(T)\big\vert f\circ a=f\circ b \big\}
\end{equation*}
and this clearly defines an equivalence relation. The morphisms $R\rightrightarrows U$ are \'etale as base changes of the \'etale morphism $f$.

We are now going to prove $U/R\simeq X$. By \cite[Tag 086K]{stacks-project} we only need to show that $U\rightarrow X$ is an epimorphism of sheaves. Since $U\rightarrow X$ is surjective, the base change $R\rightarrow U$ is surjective as well and this is equivalent to $h_R\rightarrow h_U$ being an epimorphism of \'etale sheaves, since both $U$ and $R$ are analytic spaces. Since $U\rightarrow X$ is an \'etale cover of $X$, this observation already implies the claim.
\end{proof}

\begin{remark}
Suppose that an \'etale analytic space $X$ admits a presentation by an \'etale equivalence relation $R\rightrightarrows U$ such that the diagonal $R\rightarrow U\times U$ is a closed immersion. Then by \cite[Theorem 1.2.2]{ConradTemkin_algspaces} the \'etale analytic space $X=U/R$ is representable by an analytic space. This means that all separated \'etale analytic spaces are already analytic spaces.
\end{remark}

%%%%%%%%%%%%%%%%%%%%%%%%%%%%%%%%%%%%%%%%%%%%%%%%%%%%%%

\subsubsection{Analytic groupoids}

\begin{definition}
An \emph{analytic groupoid} is a groupoid object in the category of \'etale analytic spaces, i.e. a septuple $(U,R,s,t,c,i,e)$ consisting of two \'etale analytic spaces $U$ and $R$, as well as
\begin{itemize}
\item a \emph{source morphism} $s\mathrel{\mathop:}R\rightarrow U$,
\item a \emph{target morphism} $t\mathrel{\mathop:}R\rightarrow U$, 
\item a \emph{composition morphism} $c\mathrel{\mathop:}R\times_{s,U,t}R\rightarrow R$, 
\item an \emph{inverse morphism} $i\mathrel{\mathop:}R\rightarrow R$, and
\item a \emph{unit morphism} $e\mathrel{\mathop:}U\rightarrow R$
\end{itemize}
such that for all analytic spaces $T$ over $k$ the septuple $\big(U(T),R(T),s,t,c,i,e\big)$ is a groupoid category.
\end{definition}

Note that the inverse morphism $i$ and the unit morphism $e$ are uniquely determined by $s$, $t$, and $c$. In our notation we are going to suppress the reference to the morphisms $c$, $i$, and $e$ and simply write $(s,t\mathrel{\mathop:}R\rightrightarrows U)$ or $(R\rightrightarrows U)$ for an analytic groupoid $(U,R,s,t,c,i,e)$.

An analytic groupoid $(R\rightrightarrows U)$ gives rise to a presheaf
\begin{equation*}\begin{split}
(\mathbf{An}_k)^{op}&\longrightarrow\mathbf{Groupoids} \\
T&\longmapsto \big(U(T)\rightrightarrows R(T)\big) 
\end{split}\end{equation*} 
which by \cite[Tag 0049]{stacks-project} corresponds to a category fibered in groupoids $[U/_{pre}R]$ over $\mathbf{An}_k$. In fact, the category fibered in groupoids $[U/_{pre}R]$ is pre-stack over $(\mathbf{An}_k)_{et}$.

\begin{definition}
Let $(R\rightrightarrows U)$ be an analytic groupoid. The \emph{quotient stack} $[U/R]$ is defined to be the stackification of the pre-stack $[U/_{pre}R]$.
\end{definition}

Let now $\mathcal{P}$ be a property of morphisms in $\mathbf{An}_k$ that is stable under base change and can be checked on \'etale coverings of the target and the domain. An analytic groupoid $(U,R,s,t,c)$ is said to have property $\mathcal{P}$, if the source and the target morphism $(s,t\mathrel{\mathop:}R\rightrightarrows U)$ both have property $\mathcal{P}$. It is enough to check such properties for one of the two morphisms, since the inverse morphism $i\mathrel{\mathop:}R\rightarrow R$ is an isomorphism. 

\begin{proposition}\label{prop_groupoidquot=analstack}
Let $(R\rightrightarrows U)$ be a $G$-smooth, surjective, and universally submersive analytic groupoid. Then:
\begin{enumerate}[(i)]
\item The quotient stack $\mathcal{X}=[U/R]$ is an analytic stack.  
\item If the groupoid $(R\rightrightarrows U)$ is \'etale, the quotient stack $\calX=[X/R]$ is an analytic Deligne-Mumford stack.
\end{enumerate}
\end{proposition}

\begin{example} 
Let $G$ be an analytic group acting on an analytic space $X$. Then we have an analytic groupoid  $(G\times X\rightrightarrows X)$ given as follows:
\begin{itemize}
\item the source morphism $s\mathrel{\mathop:}G\times X\rightarrow X$ by $(g,x)\mapsto x$,
\item the target morphism $t\mathrel{\mathop:}G\times X\rightarrow X$ by $(g,x)\mapsto g\cdot x$,
\item the composition morphism $c\mathrel{\mathop:}(G\times X)\times_X (G\times X)\rightarrow (G\times X)$ by $\big((g,x),(g',x')\big)\mapsto (g'g,x)$,
\item the inverse morphism $i\mathrel{\mathop:}G\times X\rightarrow G\times X$ by $(g,x)\mapsto (g^{-1},x)$, and
\item the unit morphism by $e \mathrel{\mathop:}X\rightarrow G\times$ by $x\mapsto (1,x)$.
\end{itemize}
In this case the quotient stack will be denoted by $[X/G]$. By Proposition \ref{prop_groupoidquot=analstack}  the quotient $[X/G]$ is an analytic stack and, if $G$ is finite, it is an analytic Deligne-Mumford stack.
\end{example}

The proof of Proposition \ref{prop_groupoidquot=analstack} below is an adaption of the proof of \cite[Tag 04TK]{stacks-project} to the non-Archimedean analytic situation. 

\begin{lemma}\label{lemma_quotrepdiag}
Let  $(R\rightrightarrows U)$ be an analytic groupoid. Then the diagonal morphism $\Delta_\mathcal{X}\mathrel{\mathop:}\mathcal{X}\rightarrow\mathcal{X}\times\mathcal{X}$ of the quotient stack $\mathcal{X}=[U/R]$ is representable by \'etale analytic spaces.
\end{lemma}

\begin{proof}
By Lemma \ref{lemma_diag=rep} we only need to show that for an analytic space $T$ and two objects $x,y\in\mathcal{X}(T)$ the sheaf $\Isom_\mathcal{X}(x,y)$ is representable. 
We claim that there is an \'etale covering $(T_i\rightarrow T)$ such that the sheaf $\Isom_\mathcal{X}(x,y)\vert_{T_i}$ is representable by an \'etale analytic space. 
In order to see this we can choose the $T_i$ small enough so that we may assume that $\mathcal{X}\vert_{T_i}=[U/_{pre}R]\vert_{T_i}$ by the universal property of stackification. 
In this case we have a cartesian diagram
\begin{equation*}\begin{CD}
\Isom_\mathcal{X}(x,y)\vert_{T_i}@>>>R\\
@VVV @VV(s,t)V\\
T_i@>\big(x\vert_{T_i},y\vert_{T_i}\big)>> U\times U
\end{CD}\end{equation*} 
and this shows that $\Isom_\mathcal{X}(x,y)\vert_{T_i}$ is representable by an \'etale analytic space. 
Therefore, since $\Isom_\calX(x,y)$ is a sheaf in the \'etale topology, it is representable by an \'etale analytic space itself.
\end{proof}

\begin{lemma}\label{lemma_groupoidquot=cartesian}
Let $(R\rightrightarrows U)$ be an analytic groupoid. Then the natural square
\begin{equation*}\begin{CD}
R@>s>> U\\
@VtVV @VVV\\
U@>>>[U/R]
\end{CD}\end{equation*}
is $2$-cartesian.
\end{lemma}

\begin{proof}
Let $T$ be an analytic space and consider two elements $x$ and $y$ in $[U/R](T)$. Choose an \'etale covering $(T_i\rightarrow T)$ of $T$ such that $[U/R]\vert_{T_i}=[U/_{pre}R]\vert_{T_i}$. As above we have again a cartesian diagram
\begin{equation*}\begin{CD}
\Isom_\mathcal{X}(x,y)\vert_{T_i}@>>>R\\
@VVV @VV(s,t)V\\
T_i@>\big(x\vert_{T_i},y\vert_{T_i}\big)>> U\times U
\end{CD}\end{equation*}
and this shows $\Isom_\mathcal{X}(x,y)\vert_{T_i}\simeq (T\times_{U\times U} R)\vert_{T_i}$. Since both sides of this equation are sheaves, we obtain a global isomorphism 
\begin{equation*}
\Isom_\mathcal{X}(x,y)\simeq T\times_{U\times U} R
\end{equation*}
and this immediately implies that the natural functor $R\rightarrow U\times_{[U/R]}U$ is an equivalence. 
\end{proof}

\begin{proof}[Proof of Proposition \ref{prop_groupoidquot=analstack}]
By Lemma \ref{lemma_quotrepdiag} the diagonal morphism of $[U/R]$ is representable by \'etale analytic spaces. We need to check that $U\rightarrow [U/R]$ is $G$-smooth, surjective, and universally submersive. For this let $T\rightarrow[U/R]$ be a morphism from an analytic space $T$ into the quotient stack $[U/R]$. It is enough to check these properties \'etale locally on $T$. So take an \'etale cover $(T_i\rightarrow T)$ of $T$ such that $[U/R]\vert_{T_i}\simeq[U/_{pre}R]\vert_{T_i}$ and we can assume that $T_i\rightarrow\calX$ comes from a morphism $x_i\mathrel{\mathop:}T_i\rightarrow U$. 
In this case, by Lemma \ref{lemma_groupoidquot=cartesian}, there are natural equivalences
\begin{equation*}
U\times_{[U/R]}T_i\simeq (U\times_{[U/R]}U)\times_{s,U,x_i} T_i\simeq R\times_{s,U,x_i} T_i
\end{equation*}
and the projection morphism $R\times_U T_i\rightarrow T_i$ is $G$-smooth, surjective, and universally submersive as a base change of $s\mathrel{\mathop:}R\rightarrow U$.
In the case that $s$ is \'etale, the morphism $U\times_{[U/R]}T\rightarrow T$ is \'etale and $[U/R]$ is an analytic Deligne-Mumford stack. 
\end{proof}

A \emph{groupoid presentation} of an analytic stack $\mathcal{X}$ consists of an analytic groupoid $(R\rightrightarrows U)$ together with an equivalence $[U/R]\simeq\mathcal{X}$. Following the construction presented in \cite[Tag 04T3]{stacks-project} one can show that every analytic stack $\calX$ has a $G$-smooth, surjective, and universally submersive groupoid presentation; it is given by a $G$-smooth, surjective, universally submersive atlas $U$ of $\calX$ and $R=U\times_\calX U$.

\subsection{Analytification}\label{section_analytification}

As explained in \cite[Theorem 3.4.1 and Theorem 3.5.1]{Berkovich_book} and \cite[Proposition 2.6.1]{Berkovich_etalecoho} there is an \emph{analytification functor} 
\begin{equation}\begin{split}\label{eq_analytificationschemes}
(.)^{an}\mathrel{\mathop:}\mathbf{Sch}_{loc.f.t./k}&\longrightarrow \mathbf{An}_k\\
X&\longmapsto X^{an}
\end{split}\end{equation}
from the category of schemes locally of finite type over $k$ into the category of $k$-analytic spaces that respects fiber products and therefore all finite limits. By \cite[Proposition 3.3.11]{Berkovich_etalecoho} algebraic \'etale morphisms on the left side of \eqref{eq_analytificationschemes} induce analytic \'etale morphisms on the right side. Thus $(.)^{an}$ is a continuous functor with respect to the \'etale topologies and it therefore defines a morphism
\begin{equation*}
\alpha\mathrel{\mathop:}(\mathbf{An}_k)_{et}\longrightarrow \big(\mathbf{Sch}_{loc.f.t./k}\big)_{et}
\end{equation*}
from the analytic to the algebraic \'etale site. 

\begin{definition}
Given an algebraic stack $\calX$ locally of finite type over $k$ we define its \emph{associated analytic stack} $\calX^{an}$ as the pullback $\alpha^\ast\calX$ of $\calX$ along $\alpha$ in the sense of \cite[Section II.3.2]{Giraud_cohomologienonabelienne} and \cite[Tag 04WJ]{stacks-project}.
\end{definition}

A priori the pullback $\alpha^\ast\calX$ is only a stack; we will see in Corollary \ref{cor_analytificationisanalytic} below that $\calX^{an}=\alpha^\ast\calX$ is analytic. Let $\calY$ be an analytic stack. Then by the universal property of pullback there is natural equivalence
\begin{equation}\label{eq_upan}
\HOM(\calY,\calX^{an})\simeq\HOM(\alpha_\ast\calY,\calX)
\end{equation}
between the functor categories, where $\alpha_\ast\calY$ is the pushforward of $\calY$ along $\alpha$, i.e. the restriction of $\calY$ to $\mathbf{Sch}_{loc.f.t./k}$ along $\alpha$. For a scheme $X$, locally of finite type over $k$, the pullback $\alpha^\ast X$ is nothing but the \emph{analytic space} $X^{an}$ associated to $X$. So, given an analytic space $Y$, the equivalence \eqref{eq_upan} reduces to a  bijection 
\begin{equation*}
\Hom(Y,X^{an})\simeq\Hom(\alpha_\ast Y,X)=\Hom(Y,X) \ ,
\end{equation*}
i.e. to the universal property of the analytification functor $(.)^{an}$ in \cite[Theorem 3.4.1 and 3.5.1]{Berkovich_book}.

By \cite[Tag 00XS]{stacks-project} taking pullbacks commutes with coequalizers of sheaves and therefore for an \'etale equivalence relation $R$ on an analytic space $U$ there is a natural isomorphism $U^{an}/R^{an}\simeq(U/R)^{an}$. This shows that the analytification of an algebraic space, locally of finite type over $k$, is an \'etale analytic space, since by \cite[Tag 0262]{stacks-project} every algebraic space $X$ has a presentation by an \'etale equivalence relation. 

The following Proposition \ref{prop_analytificationgroupoid} shows that $(.)^{an}$ respects groupoid quotients. 

\begin{proposition}\label{prop_analytificationgroupoid}
Let $\calX$ be an algebraic stack locally of finite type over $k$ and $[U/R]\simeq \calX$ a groupoid presentation of $\calX$ of by algebraic spaces locally of finite type over $k$.
Then there is a natural equivalence
\begin{equation*}
\calX^{an}\simeq[U^{an}/R^{an}] \ .
\end{equation*}
\end{proposition}

Using Proposition \ref{prop_analytificationgroupoid} we can show that the analytification of an algebraic stack locally of finite type over $k$ is an analytic stack. 

\begin{corollary}\label{cor_analytificationisanalytic}
Let $\calX$ be an algebraic stack that is locally of finite type over $k$. Then the stack $\calX^{an}$ is analytic. Moreover, if $\calX$ is a Deligne-Mumford stack, then $\calX^{an}$ is an analytic Deligne-Mumford stack. 
\end{corollary}

\begin{proof}
By \cite[Tag 04T3]{stacks-project} every algebraic stack locally of finite type over $k$ has a smooth and surjective groupoid presentation $[U/R]\simeq\calX$ in the category of algebraic spaces that are locally of finite type over $k$. By \cite[Proposition 3.4.6]{Berkovich_book} $(R^{an}\rightrightarrows U^{an})$ is surjective and by \cite[Proposition 3.5.8]{Berkovich_etalecoho} smooth in the sense of \cite[Section 3.5]{Berkovich_etalecoho}. Smooth analytic morphisms are stable under base change by \cite[Proposition 3.5.2]{Berkovich_etalecoho} and open by \cite[Corollary 3.7.4]{Berkovich_etalecoho}. Therefore both $s^{an}$ and $t^{an}$ are universally submersive. Moreover, smooth morphisms are $G$-smooth and therefore $\calX^{an}$ is analytic by Proposition \ref{prop_groupoidquot=analstack} (i).

If $\calX$ is a Deligne-Mumford stack, we can find an \'etale and surjective groupoid presentation $[U/R]\simeq \calX$ of $\calX$ by algebraic spaces locally of finite type over $k$. In this case the induced analytic groupoid $(R^{an}\rightrightarrows U^{an})$ is \'etale by \cite[Proposition 3.3.11]{Berkovich_etalecoho} and surjective by \cite[Proposition 3.4.6]{Berkovich_book}. Therefore $\calX^{an}=[U^{an}/R^{an}]$ is an analytic Deligne-Mumford stack by Proposition \ref{prop_groupoidquot=analstack} (ii).
\end{proof}

The general properties of pullbacks (see \cite[Section II.3.2]{Giraud_cohomologienonabelienne} and \cite[Tag 04WJ]{stacks-project}) ensure that there is an \emph{analytification pseudofunctor} 
\begin{equation*}\begin{split}
(.)^{an}\mathrel{\mathop:}\mathbf{Alg.Stacks}_{loc.f.t./k}&\longrightarrow \mathbf{An.Stacks}_k \\
\calX&\longmapsto \calX^{an}
\end{split}\end{equation*} 
that restricts to the usual analytification functor on the full subcategory of schemes locally of finite type over $k$. This functor is unique up to equivalence.

\begin{example}
Let $G$ be an algebraic group acting on a scheme $X$ that is locally of finite type over $k$. Then the analytification $[X/G]^{an}$ of the quotient stack $[X/G]$ is given by $[X^{an}/G^{an}]$.
\end{example}

The rest of this section is devoted to the proof of Proposition \ref{prop_analytificationgroupoid}.

\begin{proof}[Proof of Proposition \ref{prop_analytificationgroupoid}]
This proof follows ideas of the proof of \cite[Tag 04WX]{stacks-project}. Let us first recall the construction of $\alpha^\ast \calX$ in our situation. 
Consider the category $\calX^{an,pp}$ over $\mathbf{An}_k$ defined as follows:
\begin{itemize}
\item An object of $\calX^{an,pp}$ is a triple $(T,\phi\mathrel{\mathop:}T'\rightarrow T^{an},x)$, where $T$ is an object of $\mathbf{Sch}_{loc.f.t./k}$, the arrow $\phi$ is a morphism in $\mathbf{An}_k$ and $x\mathrel{\mathop:}T\rightarrow U$ is morphism of schemes. 
\item A morphism 
\begin{equation*}
(a,a',\gamma)\mathrel{\mathop:}(T_1,\phi_1\mathrel{\mathop:}T'_1\rightarrow T_1^{an},x_1)\longrightarrow (T_2,\phi\mathrel{\mathop:}T'_2\rightarrow T_2^{an},x_2)
\end{equation*}
consists of a morphism $a\mathrel{\mathop:}T_1\rightarrow T_2$ and a morphism $a'\mathrel{\mathop:}T'_1\rightarrow T'_2$ in $\mathbf{An}_k$ such that the diagram
\begin{equation*}\begin{CD}
T'_1@>a'>> T'_2\\
@VVV @VVV\\
T_1^{an}@>a^{an}>> T_2^{an}
\end{CD}\end{equation*}
commutes, as well as a morphism $\gamma\mathrel{\mathop:}T_1\rightarrow R$ such that the diagram
\begin{center}\begin{tikzpicture}
  \matrix (m) [matrix of math nodes,row sep=3em,column sep=4em,minimum width=2em]
  { & U\\
     T_1 & R \\
     T_2& U \\};
  \path[-stealth]
    (m-2-1) edge node [left] {$a$} (m-3-1)
            edge node [above] {$\gamma$} (m-2-2)
             edge node [above] {$x_1$} (m-1-2)
    (m-3-1) edge node [above] {$x_2$} (m-3-2)
    (m-2-2) edge node [right] {$s$} (m-1-2)
    (m-2-2) edge node [right] {$t$} (m-3-2);
\end{tikzpicture}\end{center}
commutes.
\item The functor $\calX^{an,pp}\rightarrow \mathbf{An}_k$ is given by 
\begin{equation*}
(T,\phi\mathrel{\mathop:}T'\rightarrow T^{an},x)\longmapsto T' \ .
\end{equation*}
\end{itemize}
Now let $S$ denotes the set of arrows in $\calX^{an,pp}$ of the from 
\begin{equation*}
(a,\id_{T'},\gamma)\mathrel{\mathop:}(T_1,\phi_1\mathrel{\mathop:}T'\rightarrow T_1^{an},x_1)\longrightarrow (T_2,\phi_2\mathrel{\mathop:}T'\rightarrow T_2^{an},x_2)
\end{equation*}
such that $\gamma$ is strongly cartesian as a morphism in $[U/_{pre}R]$ over $\mathbf{Sch}_{loc.f.t./k}$. By \cite[Tag 04WF]{stacks-project} the set $S$ is right-multiplicative and by \cite[Tag 04WG and Tag 04WH]{stacks-project} the localization $\calX^{an,p}=S^{-1}\calX^{an}$ is a category fibered in groupoids over $\mathbf{An}_k$. As defined in \cite[Tag 04WJ]{stacks-project} the analytification $\calX^{an}$ is the stackification of $\calX^{an,p}$. 

Having developed this terminology we can now prove our claim. Define a functor $[U/_{pre}R]^{an,pp}\rightarrow [U^{an}/_{pre}R^{an}]$ by
\begin{equation*}
(T,\phi\mathrel{\mathop:}T'\rightarrow T^{an},x)\longmapsto (x'=x\circ\phi \mathrel{\mathop:}T'\rightarrow U^{an})
\end{equation*}
on objects and 
\begin{equation*}
(a,a',\gamma)\longmapsto \big((a'\mathrel{\mathop:}T'_1\rightarrow T'_2),(\gamma\circ\phi_1\mathrel{\mathop:} T_1'\rightarrow R^{an})\big)
\end{equation*}
on morphisms. Since $(R^{an}\rightrightarrows U^{an})$ is a groupoid in \'etale analytic spaces, this functor sends morphisms in $R$ to isomorphisms and therefore it canonically factors through a functor
\begin{equation*}
[U/_{pre}R]^{an,p}\rightarrow [U^{an}/_{pre} R^{an}] \ .
\end{equation*}
By \cite[Tag 04WR]{stacks-project} taking pullbacks commutes with stackification and so we obtain a natural functor 
\begin{equation*}
\calX^{an}\longrightarrow [U^{an}/R^{an}]  
\end{equation*}
by the universal property of stackification.

Finally we need to prove that this functor is an equivalence; by \cite[Tag 046N]{stacks-project} it is enough to show that it is fully faithful and \'etale locally essentially surjective. The latter assertion immediately follows from $U$ admitting a surjective \'etale morphism from a scheme locally of finite type over $k$. Since $R$ also admits a surjective \'etale morphism from a scheme locally of finite type over $k$, the above functor is \'etale locally full. Moreover, for an analytic space $T'$ the images of two morphisms in $[X/_{pre}R]^{an,pp}(T')$ agree in $[U^{an}/_{pre}R^{an}]$ if and only if they differ by an element of $R(T)$. These two observations are enough to show that the above functor is full and faithful by \cite[Tag 04WQ]{stacks-project}. 
 \end{proof}

\begin{remarks}\begin{enumerate}[(i)]
\item Given a presentation $[U/R]\simeq\calX$ of an algebraic stack $\calX$ locally of finite type over $k$ by algebraic spaces locally of finite type over $k$, one could directly define $\calX^{an}$ as the groupoid quotient $[U^{an}/R^{an}]$ and show that this definition gives rise to a well-defined object. %The techniques developed in Section \ref{section_Morita} and their algebraic analogues could then be used to make the association $\calX\mapsto\calX^{an}$ into a functor that is well-defined up to equivalence. 
\item Let $X$ be a separated algebraic space locally of finite type over $k$. In \cite[Theorem 1.2.1]{ConradTemkin_algspaces} the authors show that the analytification $X^{an}$ of $X$, which is a priori only an \'etale analytic space, is representable by an analytic space. 
\end{enumerate}\end{remarks}

%%%%%%%%%%%%%%%%%%%%%%%%%%%%%%%%%%%%%%%%%%%%%%%%%%%%%%

\section{Topology of analytic stacks}\label{section_topology}

In this section we are going to define and study the functor
\begin{equation*}
\vert.\vert\mathrel{\mathop:}\mathbf{An.Stacks}_k\longrightarrow \mathbf{Top}
\end{equation*} 
that associates to an analytic stack its underlying topological space. Many results in this section are analogues of the corresponding results in the algebraic setting, as developed e.g. in \cite[Tag 04XE]{stacks-project}. 

%%%%%%%%%%%%%%%%%%%%%%%%%%%%%%%%%%%%%%%%%%%%%%%%%%%%%%

\subsection{Points of analytic stacks}\label{section_ptsanstacks} Throughout this section we fix an analytic stack $\calX$. Consider pairs $(K,p)$ consisting of a non-Archimedean field extension $K$ of $k$ and a morphism $p\mathrel{\mathop:}\mathcal{M}(K)\rightarrow\mathcal{X}$ over $k$. Two such pairs $(K,p)$ and $(L,q)$ are said to be \emph{equivalent}, if there is a non-Archimedean field extension $\Omega$ of both $K$ and $L$ making the diagram 
\begin{equation*}\begin{CD}
\mathcal{M}(\Omega)@>>>\mathcal{M}(L)\\
@VVV @VVqV\\
\mathcal{M}(K)@>>p>\mathcal{X}
\end{CD}\end{equation*}
$2$-commutative. An argument analogous to the one in \cite[Tag 04XF]{stacks-project} shows that this notion defines an equivalence relation. 

\begin{definition}
The set of \emph{points} $\vert\calX\vert$ of $\mathcal{X}$ is the set of equivalence classes of pairs $(K,p)$ as above. 
\end{definition}

If $\mathcal{X}$ is represented by an analytic space $X$ the set $\vert \mathcal{X}\vert$ recovers exactly the set $\vert X\vert$ underlying $X$. A morphism $f\mathrel{\mathop:}\mathcal{X}\rightarrow\mathcal{Y}$ of analytic stacks induces a well-defined map $\vert f\vert\mathrel{\mathop:}\vert\mathcal{X}\vert\rightarrow\vert\mathcal{Y}\vert$ that is given by sending a representative $(K,p)$ of a point in $\vert\mathcal{X}\vert$ to the composition $(K,f\circ p)$. Moreover, the association $f\mapsto\vert f\vert$ is functorial. Note, in particular, that, given a $2$-commutative square
\begin{equation*}\begin{CD}
\mathcal{W}@>>>\mathcal{X}\\
@VVV @VVV\\
\mathcal{Y}@>>>\mathcal{Z}
\end{CD}\end{equation*}
of analytic stacks, the induced diagram
\begin{equation*}\begin{CD}
\vert\mathcal{W}\vert @>>>\vert\mathcal{X}\vert\\
@VVV @VVV\\
\vert\mathcal{Y}\vert @>>>\vert\mathcal{Z}\vert
\end{CD}\end{equation*}
is commutative in the category of sets.  

\begin{lemma}\label{lemma_morphismstopspace}
\begin{enumerate}[(i)]
\item An equivalence $\mathcal{X}\rightarrow\mathcal{Y}$ of analytic stacks induces a natural bijection $\vert \mathcal{X}\vert\xrightarrow{\sim}\vert \mathcal{Y}\vert$.
\item Let $\mathcal{X}\rightarrow\mathcal{Z}$ and $\mathcal{Y}\rightarrow\mathcal{Z}$ be morphisms of analytic stacks. Then the induced map $\vert \mathcal{X}\times_\mathcal{Z}\mathcal{Y}\vert\rightarrow\vert\mathcal{X}\vert\times_{\vert\mathcal{Z}\vert}\vert\mathcal{Y}\vert$ is surjective.
\item Let $f\mathrel{\mathop:}\mathcal{X}\rightarrow\mathcal{Y}$ be a morphism of analytic stacks that is representable by \'etale analytic spaces. Then $f$ is surjective if and only if $\vert f\vert\mathrel{\mathop:}\vert\mathcal{X}\vert\rightarrow\vert\mathcal{Y}\vert$ is surjective.  
%\item If $f\mathrel{\mathop:}\calX\rightarrow\calY$ is a monomorphism of analytic stacks, then $\vert f\vert\mathrel{\mathop:}\vert\mathcal{X}\vert\rightarrow\vert\mathcal{Y}\vert$ is injective.
\end{enumerate}
\end{lemma}

Our proof of Lemma \ref{lemma_morphismstopspace} is a simple adaptation of the proofs of the corresponding statements in \cite[Tag 04XE]{stacks-project} and \cite[Tag 0500]{stacks-project}.

\begin{proof}[Proof of Lemma \ref{lemma_morphismstopspace}]
Part (i) immediately follows from the above reasoning, since two naturally equivalent morphisms induce the same morphism on the underlying topological spaces. 

The proof of part (ii) is word-by-word the same as the proof of \cite[Tag 04XH]{stacks-project}. Let $K$ and $L$ be two non-Archimedean extensions of $k$ and consider two morphisms $\mathcal{M}(K)\rightarrow\mathcal{X}$ and $\mathcal{M}(L)\rightarrow\mathcal{Y}$, whose compositions $\mathcal{M}(K)\rightarrow\mathcal{X}\rightarrow\mathcal{Z}$ and $\mathcal{M}(L)\rightarrow\mathcal{Y}\rightarrow\mathcal{Z}$ are equal as elements of $\vert\calZ\vert$. Then there is a common non-Archimedean extension $\Omega$ of both $K$ and $L$ such that $\calM(\Omega)\rightarrow \calZ$ and $\calM(\Omega)\rightarrow \calZ$ are $2$-isomorphic. But this is exactly the datum of a morphism $\mathcal{M}(\Omega)\rightarrow\mathcal{X}\times_\mathcal{Z}\mathcal{Y}$.

For part (iii) suppose first that $\vert f\vert\mathrel{\mathop:}\vert\mathcal{X}\vert\rightarrow\vert\mathcal{Y}\vert$ is surjective. Let $T\rightarrow\mathcal{Y}$ be a morphism from an analytic space $T$ to $\mathcal{Y}$ and $S\rightarrow \calX\times_\calY T$ a surjective morphism from an analytic space $S$ onto $\calX\times_\calY T$. Then the map $\vert S\vert\rightarrow\vert T\vert$ factors as $\vert S\vert\rightarrow\vert\mathcal{X}\times_\mathcal{Y}T\vert\rightarrow\vert\mathcal{X}\vert\times_{\vert\mathcal{Y}\vert}\vert T\vert\rightarrow\vert T\vert$ and is therefore surjective by part (ii). 

Conversely assume that $f\mathrel{\mathop:}\mathcal{X}\rightarrow\mathcal{Y}$ is surjective. Then, given a pair $(K,p)$ consisting of a non-Archimedean extension $K$ of $k$ and a morphism $p\mathrel{\mathop:}\mathcal{M}(K)\rightarrow\mathcal{Y}$, the induced morphism $\mathcal{X}\times_\mathcal{Y}\mathcal{M}(K)\rightarrow\mathcal{M}(K)$ is surjective as a morphism of \'etale analytic spaces. Let $S\rightarrow \calX\times_\calY\calM(K)$ be a surjective morphism from an analytic space $S$. Since $S\rightarrow \calM(K)$ is surjective, we can find a pair $(K',p')$ consisting of a non-Archimedean extension $K'$ of $K$ and a morphism $p'\mathrel{\mathop:}\calM(K')\rightarrow S$ such that the induced composition 
\begin{equation*}
\mathcal{M}(K')\longrightarrow S\longrightarrow \mathcal{X}\times_\mathcal{Y}\mathcal{M}(K)\longrightarrow\mathcal{M}(K)
\end{equation*}
is the morphism induced by $K\hookrightarrow K'$. This proves that $\vert f\vert$ is surjective.
 
%Finally for part (iv) suppose that $f\mathrel{\mathop:}\calX\rightarrow\calY$ is a monomorphism. Given two points $x_1$ and $x_2$ in $\vert \calX\vert$ represented by morphisms $\underline{x}_i\mathrel{\mathop:}\calM(K_i)\rightarrow\calX$ such that $\vert f\vert(x_1)=\vert f\vert (x_2)$, we can find morphisms $c_i\mathrel{\mathop:}\calM(\Omega)\rightarrow\calM(K_i)$ from a common non-Archimedean extension $\Omega $ of both $K_1$ and $K_2$ such that $f\circ\underline{x_1}\circ c_1$ and $f\circ\underline{x_2}\circ c_2$ are equal up to $2$-isomorphism. Since $f$ is a monomorphism, i.e. a fully faithful functor, we have that also $\underline{x}_1\circ c_1$ and $\underline{x}_2\circ c_2$ are equal up to $2$-isomorphism, which shows $x_1=x_2$ in $\vert \calX\vert$. 
 \end{proof}

\begin{definition}
Let $\mathcal{X}$ be an analytic stack and choose a surjective universally submersive morphism $U\rightarrow\mathcal{X}$ from an analytic space $U$ onto $\mathcal{X}$. The set $\vert\mathcal{X}\vert$ endowed with the quotient topology induced via $\vert U\vert\rightarrow\vert \mathcal{X}\vert$ is called the \emph{topological space underlying} $\mathcal{X}$. 
\end{definition}

The topology on $\vert\mathcal{X}\vert$ does not depend on the choice of an atlas $U\rightarrow\mathcal{X}$ by the following Proposition \ref{prop_topanalstack}. 

\begin{proposition}\label{prop_topanalstack} Let $\mathcal{X}$ be an analytic stack.
\begin{enumerate}[(i)] 
\item For every universally submersive surjective morphism $U'\rightarrow\mathcal{X}$ from an analytic space $U'$ onto $\mathcal{X}$ the induced surjective map $\vert U'\vert\rightarrow\vert\mathcal{X}\vert$ is a topological quotient map. If $U'\rightarrow \mathcal{X}$ is \'etale, the quotient map $\vert U'\vert\rightarrow\vert\mathcal{X}\vert$ is open.
\item Let $[U/R]\simeq\mathcal{X}$ be a groupoid presentation of an analytic stack $\mathcal{X}$. Then the image of $\vert R\vert\rightrightarrows\vert U\vert\times\vert U\vert$ defines an equivalence relation on $\vert U\vert$ and $\vert\mathcal{X}\vert$ is the topological quotient of $\vert U\vert$ by this equivalence relation. 
\item For every morphism $f\mathrel{\mathop:}\mathcal{X}\rightarrow\mathcal{Y}$ of analytic stacks the induced map $\vert f\vert\mathrel{\mathop:}\vert\mathcal{X}\vert\rightarrow\vert\mathcal{Y}\vert$ is continuous.  
\end{enumerate}\end{proposition}

For the proof of Proposition \ref{prop_topanalstack} we simply adapt the proof of \cite[Tag 04XL]{stacks-project} to the non-Archimedean analytic situation.

\begin{proof}
Taking the fiber product $U\times_\mathcal{X}U'$ induces a diagram
\begin{equation*}\begin{CD}
\vert U\times_\mathcal{X} U'\vert @>>>\vert U'\vert\\
@VVV @VVV\\
\vert U\vert @>>> \vert\mathcal{X}\vert
\end{CD}\end{equation*}
where the upper horizontal and the left vertical arrow are surjective topological quotient maps. This immediately implies that the surjective map $\vert U'\vert\rightarrow\vert\mathcal{X}\vert$ is also a topological quotient map.

For part (ii) we remark that by Lemma \ref{lemma_morphismstopspace} (iii) the induced map $\vert U\vert\rightarrow\vert\mathcal{X}\vert$ is surjective. Since $R\rightarrow U\times_\mathcal{X}U$ is also surjective, the induced morphism $\vert R\vert\rightarrow \vert U\vert\times_{\vert \mathcal{X}\vert}\vert U\vert$ is surjective by Lemma \ref{lemma_morphismstopspace} (ii) and (iii). Thus the image of $\vert R\vert\rightarrow \vert U\vert\times\vert U\vert$ is exactly the set of pairs $(u_1,u_2)$ consisting of elements $u_1$ and $u_2$ in $\vert U\vert$ that have the same image in $\vert \mathcal{X}\vert$, i.e. $\vert\mathcal{X}\vert$ is the set-theoretic quotient of $\vert U\vert$ by the equivalence relation $\vert R\vert\rightarrow\vert U\vert\times\vert U\vert$. This defines a topological quotient by (i).

Consider now part (iii): Take an atlas $V\rightarrow \calY$ and a representable surjective \'etale cover $U\rightarrow \calX\times_\calY V$. This gives rise to a $2$-commutative diagram 
\begin{equation*}\begin{CD}
U@>f'>> V\\
@VVV @VVV\\
\mathcal{X}@>f>>\mathcal{Y}
\end{CD}\end{equation*}
such that the vertical arrows are universally submersive surjective morphisms. But then the vertical arrows induce surjective quotient maps of the underlying topological spaces and therefore the continuity of $\vert f'\vert$ implies that $\vert f\vert$ is continuous.
\end{proof}

\begin{corollary}\label{cor_loccomplocpathconn}
The underlying topological space $\vert\mathcal{X}\vert$ of an analytic Deligne-Mumford stack $\mathcal{X}$ is locally compact and locally path-connected. 
\end{corollary}

\begin{proof}
Choose a surjective \'etale morphism $U\rightarrow \mathcal{X}$. By Proposition \ref{prop_topanalstack} (i) the quotient map $\vert U\vert\rightarrow\vert\mathcal{X}\vert$ is open and therefore $\vert\mathcal{X}\vert$ is locally compact and locally path-connected, since $\vert U\vert$ is locally compact and locally path-connected.
\end{proof}

\begin{corollary}\label{cor_fingrpquot}
Let $U$ be an analytic space and $\Gamma$ be a finite group acting analytically on $U$. Then the underlying topological space of the quotient stack $[U/\Gamma]$ is equal to $\vert U\vert /\Gamma$.
\end{corollary}

\begin{proof}
This immediately follows from Proposition \ref{prop_topanalstack} (ii).
\end{proof}

\begin{remark}
Let $G$ be an analytic group that is operating on an analytic space $X$ and let $H$ be a (not necessarily analytic) subgroup of $G$. In \cite[Section 5.1]{Berkovich_book} the author introduces a topological space $X/H$ that functions as a quotient of $X$ by $H$. Its points are precisely the orbits (in the sense of \cite[Section 5.1]{Berkovich_book}) of $H$ in $X$ and $X/H$ is endowed with the quotient topology from $X$. Therefore, if $H$ is an analytic group itself, then by Proposition \ref{prop_topanalstack} (i) the topological space $X/H$ is naturally homeomorphic to $\big\vert[X/H]\big\vert$. So, in this case $X/H$ naturally carries the structure of an analytic stack.
\end{remark}

\subsection{Topology and analytification}
Let $\mathcal{X}$ be an algebraic stack that is locally of finite type over $k$. For a non-Archimedean extension $K$ of $k$ we have a natural equivalence
\begin{equation*}
\HOM(\calM(K),\calX^{an})\simeq\HOM(\Spec K, \calX)
\end{equation*}
and therefore one may describe $\vert\calX^{an}\vert$ as the set of equivalence classes of pairs $(K,p)$ consisting of a non-Archimedean extension $K$ of $k$ and a morphism $\Spec K\rightarrow \mathcal{X}$. Two such paris $(K,p)$ and $(L,q)$ are hereby \emph{equivalent}, if there is a non-Archimedean field extension $\Omega$ of both $K$ and $L$ such that the diagram
\begin{equation*}\begin{CD}
\Spec \Omega @>>>\Spec L\\
@VVV @VVV\\
\Spec K @>>>\mathcal{X}
\end{CD}\end{equation*}
is $2$-commutative.

Now suppose in addition that $\calX$ be a separated algebraic Deligne-Mumford stack locally of finite type over $k$. By \cite[Corollary 1.3 (1)]{KeelMori_groupoidquotients} the stack $\calX$ admits a coarse moduli space $X$. Since the coarse moduli space $X$ is separated, \cite[Theorem 1.2.1]{ConradTemkin_algspaces} implies that the \'etale analytic space $X^{an}$ is in fact an analytic space (also see Remark \ref{remark_etanspacerepdiag} (ii)).

\begin{proposition}\label{prop_topologycoarsemoduli}
The topological spaces $\vert\calX^{an}\vert$ and $\vert X^{an}\vert$ are naturally homeomorphic. 
\end{proposition}

\begin{proof}
For every non-Archimedean algebraically closed field $K$ extending $k$ there is a natural equivalence $\calX(K)\simeq X(K)$ by the definition of coarse moduli spaces. Therefore the above description immediately implies that $\vert\calX^{an}\vert\xrightarrow{\sim} \vert X^{an}\vert$ is a continuous bijection. We still need to show that this map is a homeomorphism. 
 
By \cite[Lemma 2.2.3]{AbramovichVistoli_compactifyingstablemaps} there is an \'etale covering $(X_i\rightarrow X)$ of $X$ as well as a scheme $U_i$ locally of finite type over $k$ and a finite group $\Gamma_i$ such that the pullback $\calX\times_XX_i$ is equivalent to $[U_i/\Gamma_i]$. Since the $X_i\rightarrow X$ are etale, the $X_i$ are coarse moduli spaces of $\calX\times_XX_i=[U_i/\Gamma_i]$ and therefore $X_i=U_i/\Gamma_i$. By Proposition \ref{prop_analytificationgroupoid} we have $[U_i/\Gamma_i]^{an}\simeq[U_i^{an}/\Gamma_i]$ and therefore Corollary \ref{cor_fingrpquot} shows
\begin{equation*}
\big\vert [U_i/\Gamma_i]^{an}\big\vert=\big\vert U_i^{an}\big\vert/\Gamma_i=\big\vert U_i^{an}/\Gamma_i\big\vert
\end{equation*} 
on the level of the underlying topological spaces. So the morphism $\vert\calX\times_XX_i\vert\rightarrow \vert X_i\vert$ is a homeomorphism. This gives rise to commutative diagrams
 \begin{equation*}\begin{CD}
 \vert U_i^{an}\vert @>>> \vert X_i^{an}\vert\\
 @VVV @VVV\\
\vert \calX^{an}\vert @>>>\vert X^{an}\vert
 \end{CD}\end{equation*}
where both the two vertical and the upper horizontal arrow are open maps. Since the $X_i$ cover $X$ and the $U_i$ cover $\calX$, this is enough to show that the continuous bijection $\vert\calX^{an}\vert\xrightarrow{\sim}\vert X^{an}\vert$ is open and therefore a homeomorphism.
\end{proof}

\begin{remark}
Suppose that $k$ is an algebraically closed field endowed with the trivial norm. In \cite{AbramovichCaporasoPayne_tropicalmoduli} the authors show that, given a proper toroidal algebraic Deligne-Mumford stack $\calX$, the analytic space $X^{an}$ associated to its coarse moduli space $X$ admits a strong deformation retraction $\mathbf{p}_\calX$ of $X^{an}$ onto its \emph{skeleton} $\frakS(\calX)$, a closed subset of $\vert X^{an}\vert$ that has the structure of a generalized extended cone complex in the sense of \cite[Section 2]{AbramovichCaporasoPayne_tropicalmoduli}. Proposition \ref{prop_topologycoarsemoduli} tells us that $\frakS(\calX)$ naturally embeds into $\vert \calX^{an}\vert$ and $\mathbf{p}_\calX$ is actually a strong deformation retraction of $\vert\calX^{an}\vert$. 
\end{remark}

%%%%%%%%%%%%%%%%%%%%%%%%%%%%%%%%%%%%%%%%%%%%%%%%%%%%%%

\section{Skeletons and stack quotients}\label{section_trop=quot}

The goal of this section is to prove Theorem \ref{thm_trop=quot}.
Let $T\simeq \G_m^n$ be a split algebraic torus over $k$ and denote by $N$ the dual of its character lattice $M$. Suppose that $X=X(\Delta)$ is a $T$-toric variety defined by a rational polyhedral fan $\Delta$ in $N_\R=N\otimes\R$. We refer the reader to \cite{Fulton_toricvarieties} for the standard notation concerning toric varieties. 

We recall from \cite{Kajiwara_troptoric} and \cite{Payne_anallimittrop} (also see \cite[Section 5]{Rabinoff_newtonpolygon}) that the continuous and proper \emph{tropicalization map}
\begin{equation*}
\trop_\Delta\mathrel{\mathop:}X^{an}\longrightarrow N_\R(\Delta)
\end{equation*} 
from $X^{an}$ into a partial compactification $N_\R(\Delta)$ of $N_\R$ is uniquely determined by its restrictions to the $T$-invariant open affine subsets $U_\sigma$ for cones $\sigma$ in $\Delta$. In this case the codomain $N_\R(\sigma)\subseteq N_\R(\Delta)$ is the set $\Hom(S_\sigma,\Rbar)$, where $\Rbar=\big(\R\cup\{\infty\},+\big)$, and $N_\R(\sigma)$ is endowed with the topology of pointwise convergence. On $U_\sigma=\Spec k[S_\sigma]$ the tropicalization map
\begin{equation*}
\trop_\sigma\mathrel{\mathop:}U_\sigma^{an}\longrightarrow N_\R(\sigma)
\end{equation*}
is defined by associating to an element $x\in X^{an}$ the homomorphism $s\mapsto -\log\vert\chi^s\vert_x$. 

\begin{lemma}\label{lemma_trop=skeleton}
There is a strong deformation retraction $\mathbf{p}_\Delta\mathrel{\mathop:}X^{an}\rightarrow X^{an}$ onto a closed subset $\frakS(X)$ of $X^{an}$ as well as a homeomorphism $J_\Delta\mathrel{\mathop:}N_\R(\Delta)\xrightarrow{\sim} \frakS(X)$ making the diagram
\begin{center}\begin{tikzpicture}
  \matrix (m) [matrix of math nodes,row sep=1em,column sep=2em,minimum width=2em]
  {  
  & & N_\R(\Delta)  \\ 
  X^{an}&  &\\
  & &  \frakS(X) \\
  };
  \path[-stealth]
    (m-2-1) edge node [above] {$\trop_{\Delta}\ \ \ \ $} (m-1-3)
    (m-2-1) edge node [below] {$\mathbf{p}_\Delta$} (m-3-3)
    (m-1-3) edge node [right] {$J_\Delta$} (m-3-3);
\end{tikzpicture}\end{center}
commute. 
\end{lemma}

The deformation retract $\frakS(X)$ is called the \emph{non-Archimedean skeleton} of $X$. The proof of Lemma \ref{lemma_trop=skeleton} uses techniques that have originally appeared in \cite[Section 6]{Berkovich_book}. In particular, it generalizes the constructions of \cite[Section 2]{Thuillier_toroidal} to non-Archimedean ground fields $k$ that do not necessarily carry the trivial norm.

\begin{proof}[Proof of Lemma \ref{lemma_trop=skeleton}]
Consider a $T$-invariant open subset $U_\sigma=\Spec k[S_\sigma]$ for a cone $\sigma$ in $\Delta$. Given a point $x\in U_\sigma^{an}$ we define the point $\mathbf{p}_\sigma(x)$ as the seminorm on $k[S_\sigma]$ given by
\begin{equation*}
\mathbf{p}_\sigma(x)(f)=\max_{s\in S_\sigma}\vert a_s\vert \vert\chi^s\vert_x
\end{equation*}
for an element $f=\sum_{s\in S_\sigma}a_s\chi^s$ in $k[S_\sigma]$. We also define the image $J_\sigma(u)$ of an element $u\in N_\R(\sigma)=\Hom(S_\sigma,\Rbar)$ as the seminorm on $k[S_\sigma]$ given by
\begin{equation*}
J(u)(f)=\max_{s\in S_\sigma}\vert a_s\vert \exp\big(-u(s)\big)
\end{equation*}
for an element $f=\sum_{s\in S_\sigma}a_s\chi^s$ in $k[S_\sigma]$. 

One immediately verifies that $\mathbf{p_\sigma}$ is continuous, that the equality $\mathbf{p}_\sigma\circ\mathbf{p}_\sigma=\mathbf{p}_\sigma$ holds, and that $J_\sigma$ defines a homeomorphism $N_\R(\sigma)\xrightarrow{\sim} \frakS(U_\sigma)$. Moreover we can easily check that these constructions on $T$-invariant affine open patches are compatible with restrictions and we therefore obtain a global retraction $\mathbf{p}_\Delta$ as well as a global homeomorphism $J_\Delta$. 

It remains to show that there is a strong homotopy between $\mathbf{p}_\Delta$ and the identity map on $X^{an}$. This immediate generalization of the theory developed in \cite[Section 2.2]{Thuillier_toroidal} is left to the reader, since it is not relevant for the proof of Theorem \ref{thm_trop=quot}.
\end{proof}

Denote by $\mu\mathrel{\mathop:}T\times X\rightarrow X$ the operation of $T$ on the toric variety $X$. Recall that on a $T$-invariant open affine  subset $U_\sigma$ for a cone $\sigma$ in $\Delta$ this morphism is induced by the homomorphism
\begin{equation*}\begin{split}
\mu^\#\mathrel{\mathop:}k[S_\sigma]&\longrightarrow k[M]\otimes K[S_\sigma]\\
\chi^s&\longmapsto \chi^s\otimes\chi^s \ .
\end{split}\end{equation*}
Moreover, we consider the projection morphism $\pi\mathrel{\mathop:}T\times X\rightarrow X$, which is induced by the homomorphism
\begin{equation*}\begin{split}
\pi^\#\mathrel{\mathop:}k[S_\sigma]&\longrightarrow k[M]\otimes K[S_\sigma]\\
\chi^s&\longmapsto 1\otimes\chi^s \ .
\end{split}\end{equation*}

\begin{lemma}\label{lemma_otimeshatnonbeth}
For a point $x\in U_\sigma^{an}$ consider the point $\eta\hat{\otimes}x\in T^\circ\times U_\sigma^{an}$ given by the seminorm
\begin{equation*}
\vert f\vert_{\eta\hat{\otimes}x}=\max_{m\in M}\vert a_m\vert \vert f_m\vert_x
\end{equation*}
for an element $f=\sum_{m\in M}a_m \chi^m\otimes f_m\in k[M]\otimes_kk[S_\sigma]$ with unique regular functions $f_m\in k[S_\sigma]$. Then we have
\begin{equation*}
\pi^{an}(\eta\hat{\otimes}x)=x
\end{equation*}
as well as 
\begin{equation*}
\mu^{an}(\eta\hat{\otimes}x)=\bfp_\sigma(x) \ .
\end{equation*}
\end{lemma}

\begin{proof}
Let $f=\sum_{s\in S_\sigma}a_s\chi^s\in k[S_\sigma]$. Then we have 
\begin{equation*}
\vert f\vert_{\pi^{an}(\eta\hat{\otimes}x)}=\Big\vert\sum_{s\in S_\sigma}a_s1\otimes\chi^s\Big\vert_{\eta\hat{\otimes}x}=\vert1\otimes f\vert_{\eta\hat{\otimes}x}=\vert f\vert_x
\end{equation*}
as well as 
\begin{equation*}
\vert f\vert_{\mu^{an}(\eta\hat{\otimes}x)}=\Big\vert\sum_{s\in S_\sigma}a_s\chi^s\otimes\chi^s\Big\vert_{\eta\hat{\otimes}x}=\max_{s\in S_\sigma}\vert a_s \vert \vert\chi^s\vert_x=\vert f\vert_{\bfp_\sigma(x)} 
\end{equation*}
and this implies our claim.
\end{proof}

\begin{proof}[Proof of Theorem \ref{thm_trop=quot}]
By Proposition \ref{prop_topanalstack} (ii) the topological space $\big\vert[X^{an}/T^\circ] \big\vert$ is the topological colimit of the maps 
\begin{equation}\label{equation_quot=colimit}
\big(\pi^{an},\mu^{an}\mathrel{\mathop:} T^\circ \times X^{an}\rightrightarrows X^{an}\big) \ .
\end{equation}
Therefore, by Lemma \ref{lemma_trop=skeleton} we only need to show that the deformation retraction $X^{an}\rightarrow  \frakS(X)$ makes $\frakS(X)$ into a colimit of \eqref{equation_quot=colimit}. Since $\mathbf{p}_\Delta$ is determined on the $T$-invariant open affine subsets $U_\sigma$ it is enough to prove this statement for $U_\sigma$. 
\begin{itemize}
\item Let $x, x'\in U_\sigma^{an}$ and $y\in T^\circ\times U_\sigma^{an}$ such that $\pi^{an}(y)=x$ and $\mu^{an}(y)=x'$. Then we have $\bfp_\sigma(x)=\bfp_\sigma(x')$, since 
\begin{equation*}\begin{split}
\vert\chi^{s}\vert_{x'}&=\vert \chi^s\vert_{\mu^{an}(y)}=\vert\chi^s\otimes\chi^s\vert_y\\&=\vert \chi^s\otimes 1\vert_y \cdot\vert 1\otimes \chi^s\vert_y=\vert1\otimes\chi^s\vert_y\\&=\vert\chi^s\vert_{\pi^{an}(y)}=\vert \chi^s\vert_x
\end{split}\end{equation*}
for all $s\in S_\sigma$, since $\vert\chi^m\otimes 1\vert_y=1$ for all $m\in M$.
\item Given $x\in U_\sigma^{an}$ by Lemma \ref{lemma_otimeshatnonbeth} there is a point $y=\eta\hat{\otimes}x\in T^\circ\otimes U_\sigma^{an}$ such that $\pi^{an}(y)=x$ and $\mu^{an}(y)=\bfp_\sigma(x)$. Given two points $x,x'\in U_\sigma^\beth$ such that $\bfp_\sigma(x)=\bfp_\sigma(x')$, their image in $\big\vert[U_\sigma^{an}/T^\circ]\big\vert$ is therefore equal. 
\end{itemize}
Thus the skeleton $\frakS(U_\sigma)$ is the set-theoretic colimit of \eqref{equation_quot=colimit}. It is a colimit in the category of topological spaces, since $\bfp_\sigma$ is continuous and proper. 
\end{proof}

%%%%%%%%%%%%%%%%%%%%%%%%%%%%%%%%%%%%%%%%%%%%%%%%%%%%%%%%%

%%%%%%%%%%%%%%%%%%%%%%%%%%%%%%%%%%%%%%%%%%%%%%%%%%%%%%%%%

%%%%%%%%%%%%%%%%%%%%%%%%%%%%%%%%%%%%%%%%%%%%%%%%%%%%%%

%%%%%%%%%%%%%%%%%%%%%%%%%%%%%%%%%%%%%%%%%%%%%%%%%%%%%%

%%%%%%%%%%%%%%%%%%%%%%%%%%%%%%%%%%%%%%%%%%%%%%%%%%%%%%

\bibliographystyle{amsalpha}
\bibliography{biblio}{}

\end{document}